\documentclass{amsart}
\usepackage{amssymb,verbatim,amsmath,amsfonts,amsxtra}

\theoremstyle{plain}
\newtheorem{theorem}{Theorem}[section]
\newtheorem{proposition}[theorem]{Proposition}
\newtheorem{corollary}[theorem]{Corollary}
\newtheorem{lemma}[theorem]{Lemma}
\theoremstyle{definition}
\newtheorem{definition}[theorem]{Definition}

\newcommand{\aqf}{\overset{A\textrm{-q.f.}}{\sim}}
\newcommand{\apqf}{\overset{A'\textrm{-q.f.}}{\sim}}
\newcommand{\cqf}{\overset{\mathbb{C}\textrm{-q.f.}}{\sim}}
\newcommand{\C}{\mathbb{C}}
\newcommand{\E}{\mathcal{E}}
\renewcommand{\O}{\mathcal{O}}
\newcommand{\Z}{\mathbb{Z}}
\newcommand{\R}{\mathbb{R}}

\begin{document}
\title[Equivalent Representations of Toeplitz Algebras]{On Equivalence for Representations\\ of Toeplitz Algebras}
\author{Philip M. Gipson}
\address{Department of Mathematics \\ State University of New York College at Cortland \\ Cortland, NY 13045-0900}
\email{philip.gipson@cortland.edu}
\subjclass[2000]{46L05, 46L08, 46L55}
\keywords{Toeplitz Algebras, Equivalence Relations, Endomorphisms, Hilbert Modules}
\begin{abstract} Two new notions of equivalence for representations of a Toeplitz algebra $\E_n$, $n<\infty$, on a common Hilbert space are defined. Our main results apply to $C^*$-dynamics and the conjugacy of certain $*$-endomorphisms. One particular case of the relations is shown to coincide with the multiplicity of a representation. Previously known results due to Laca and Enomoto-Watatani are recovered as special cases.
\end{abstract}
\maketitle

First introduced by Cuntz in \cite{cuntz}, the Toeplitz algebra $\E_n$ is the universal $C^*$-algebra generated by $n\leq \infty$ isometries with pairwise orthogonal ranges. The Toeplitz algebras and their representations have been surprisingly pervasive in the theory of operator algebras; appearing in contexts as diverse as crossed products 
and $K$-theory.
In this paper we will introduce two families of equivalence relations (Definitions \ref{feq} and \ref{qfeq}) on the set of nondegenerate $*$-representations of $\E_n$ on a Hilbert space $H$. These families are both indexed by the unital $C^*$-subalgebras of $B(H)$. We give applications of the relations to dynamical systems and prove a relationship between conjugacy of endomorphisms and equivalence of representations of Toeplitz algebras (Theorems \ref{thm1} and \ref{thm2}). Lastly, we demonstrate that our results recover, as special cases, previously known results of Laca \cite{laca} and Enomoto-Watatani \cite{enomoto} concerning endomorphisms of type $\textrm{I}_\infty$ and $\textrm{II}_1$ factors.

\section{Preliminaries}

Our primary objects of study will be right Hilbert $C^*$-modules over unital $C^*$-algebras. Recall that an $A$-module homomorphism $\phi:X\to Y$ is \emph{adjointable} if there is another $A$-module homomorphism $\phi^*:Y\to X$ such that $\langle\phi (x),y\rangle_Y=\langle x,\phi^*(y)\rangle_X$ for all $x\in X$ and $y\in Y$. We'll denote the space of adjointable homomorphisms between $A$-modules $X$ and $Y$ by $L(X,Y)$ and set $L(X):=L(X,X)$. A homomorphism $\phi\in L(X,Y)$ is \emph{unitary} if $\phi\circ\phi^*=id_Y$ and $\phi^*\circ\phi=id_X$; in this case we say $X$ and $Y$ are \emph{unitarily equivalent} and write $X\simeq Y$.

A subset $\{x_i:i\in I\}\subset X$ is \emph{orthonormal} if $A$ is unital and $\langle x_i,x_j\rangle=\delta_{ij}1_A$ for all $i,j\in I$. An \emph{orthonormal basis} is a right $A$-module basis which is also an orthonormal set. It is routine to check that if $\{f_1,...,f_n\}$ is any finite orthonormal basis for $X$ then we have the decomposition $x=\sum_{i=1}^nf_i\langle f_i,x\rangle$ for all $x\in X$.

The \emph{free (right, Hilbert) $A$-module of rank $n$} (where $A$ is a unital $C^*$-algebra) is the finite direct sum $A^n=A\oplus ...\oplus A$ ($n$ summands) with the coordinatewise right $A$-module structure and $A$-valued inner product given by
$$\langle (a_i),(b_i)\rangle:=\sum_{i=1}^n a_i^*b_i\in A.$$
The distinguished vectors $e_1,...,e_n\in A^n$ (where $e_i$ is the vector which has $1_A$ in the $i$-th coordinate and zeros elsewhere) form an orthonormal basis for $A^n$ which we will term the \emph{standard basis}. 
 
We can make a natural identification of $L(A^n)$ with the matrix algebra $M_n(A)$, where the homomorphisms $\theta_{e_i,e_j}:x\mapsto e_i\langle e_j,x\rangle$ play the role of matrix units. Thus $V=[v_{ij}]\in M_n(A)$ has $Vx=\sum_{j=1}^n\sum_{i=1}^ne_jv_{ji}\langle e_i,x\rangle$ for all $x\in A^n$. Because of this it is straightforward to check that $U=[u_{jk}]\in M_n(A)$ is unitary if and only if $\sum_{i=1}^nu_{ij}u_{ik}^*=\delta_{jk}1_A$ and $\sum_{i=1}^nu_{ji}^*u_{li}=\delta_{jk}1_A$ for every $j,k=1,...,n$. That is to say that matrix unitaries are precisely the unitary homomorphisms, and vice-versa. If $\{f_1,...,f_n\}$ is any orthonormal basis of $A^n$ then letting $u_{ij}:=\langle e_i,f_j\rangle$ for each $i,j=1,...,n$ we have that $U:=[u_{ij}]$ is a unitary matrix in $M_n(A)$ such that $Ue_i=f_i$ for each $i=1,...,n$. 

A \emph{covariant representation} of $A^n$ on a Hilbert space $K$ is a pair $(\sigma,\pi)$ consisting of a linear map $\sigma:A^n\to B(K)$ and a nondegenerate $*$-representation $\pi:A\to B(K)$ which together satisfy the covariance condition $\sigma(xa)=\sigma(x)\pi(a)$ for all $x\in A^n$ and $a\in A$. A covariant representation $(\sigma,\pi)$ is \emph{Toeplitz} if $\sigma(x)^*\sigma(y)=\pi(\langle x,y\rangle)$ for all $x,y\in A^n$. 

\paragraph{\textbf{Note.}} Henceforth we will only consider the case when $A$ is concretely represented, i.e. $A\subset B(H)$ for some $H$, and every covariant representation is of the form $(\sigma,id)$ for some linear map $\sigma:A^n\to B(H)$. Consequently we will abuse the notation and write $\sigma$ alone instead of $(\sigma,id)$.

Suppose that $\sigma$ is a Toeplitz representation (i.e. $(\sigma,id)$ is a Toeplitz representation) then $\sigma(e_i)^*\sigma(e_j)=\langle e_i,e_j\rangle=\delta_{ij}I$ for all $i,j=1,...,n$. Thus $\{\sigma(e_i):i=1,...,n\}$ is a family of $n$ isometries with pairwise orthogonal ranges. We will always use $v_1,...,v_n$ to denote the universal generators of $\E_n$ and so $\E_n:=C^*(\{v_1,...,v_n:v_i^*v_j=\delta_{ij}I\})$. By the universal nature of the Toeplitz algebra $\E_n$, the assignments $v_i\mapsto \sigma(e_i)$ extend uniquely to a $*$-representation of $\E_n$ on $H$ which we will denote by $\omega_\sigma$.

Conversely, given a nondegenerate $*$-representation $\omega:\E_n\to B(H)$ we may define for all $x\in A^n$
$$\sigma_\omega(x):=\sum_{i=1}^n\omega(v_i)\langle e_i,x\rangle.$$
It is easily checked that $\sigma_\omega:A^n\to B(H)$ is a linear map which, together with the identity representation of $A$, satisfies the covariance condition; hence is a covariant representation of $\E_n$. By construction we have that $\sigma_\omega$ is Toeplitz. Note that $\sigma_\omega(e_i)=\omega(v_i)$ and this uniquely determines $\sigma_\omega$.
Naturally, $\sigma_{\omega_\sigma}=\sigma$ and $\omega_{\sigma_\omega}=\omega$.

\section{Equivalences}
Given a unital $C^*$-algebra $A\subset B(H)$, we have seen that the Toeplitz covariant representations of $A^n$ coincide with the nondegenerate $*$-representations of the Toeplitz algebra $\E_n$. We will now use this relationship to define two families of equivalence relations for representations of $\E_n$.

\subsection{Free Equivalence}
Our first notion of equivalence is inspired from the observation that if $\sigma$ is a Toeplitz covariant representation of $A^n$ and $U\in M_n(A)$ a unitary then $\sigma\circ U$ is also a Toeplitz covariant representation of $A^n$. 

\begin{definition}\label{feq} Let $A\subset B(H)$ be a unital $C^*$-subalgebra. Two nondegenerate representations $\omega$ and $\tau$ of $\E_n$ (on $H$) are \emph{$A$-free equivalent} if there is a unitary homomorphism $U\in L(A^n)$ such that 
$\sigma_\omega=\sigma_\tau\circ U$.
\end{definition}
\begin{proposition} $A$-free equivalence is an equivalence relation.
\end{proposition}
\begin{proof}
Reflexivity is obvious.

Whenever $\sigma_\omega(x)=\sigma_\tau(Ux)$ for some unitary $U\in L(A^n)$ and all $x\in A^n$ we have $\sigma_\omega(U^*x)=\sigma_\tau(UU^*x)=\sigma_\tau(x)$ for all $x\in A^n$ and so $A$-free equivalence is symmetric.

Lastly, if $\sigma_\omega(x)=\sigma_\tau(Ux)$ and $\sigma_\tau(y)=\sigma_\kappa(Wy)$ for some unitaries $U,W\in L(A^n)$ and all $x,y\in A^n$ then $\sigma_\omega(x)=\sigma_\kappa(WUx)$ for all $x\in A^n$ and $WU$ is obviously a unitary in $M_n(A)$. Hence $\omega$ and $\kappa$ are $A$-free equivalent and we have transitivity.
\end{proof}

\begin{proposition}\label{feform} Two representations $\omega$ and $\tau$ are $A$-free equivalent if and only if there is a unitary $U=[u_{jk}]\in M_n(A)$ for which
$$\omega(v_i)=\sum_{j=1}^n\tau(v_j)u_{ji}$$
for all $i=1,...,n$.
\end{proposition}
\begin{proof} Suppose that there is a unitary $U\in L(A^n)$ for which $\sigma_\omega(x)=\sigma_\tau(Ux)$ for all $x\in A^n$. Since we may identify $L(A^n)$ with $M_n(A)$ and thus there is a matrix $[u_{jk}]\in M_n(A)$ such that $Ux=\sum_{j=1}^n\sum_{k=1}^ne_ju_{jk}\langle e_k,x\rangle$ for all $x\in A^n$. In particular, we have that $Ue_i=\sum_{j=1}^ne_ju_{ji}$ for all $i=1,...,n$. Hence
$$\omega(v_i)=\sigma_\omega(e_i)=\sigma_\tau(Ue_i)=\sigma_\tau\left(\sum_{j=1}^ne_ju_{ji}\right)=\sum_{j=1}^n\sigma_\tau(e_j)u_{ji}=\sum_{j=1}^n\tau(v_j)u_{ji}$$
as desired.

Conversely, if there is a unitary $U=[u_{jk}]\in M_n(A)$ for which $\omega(v_i)=\sum_{j=1}^n\tau(v_j)u_{ji}$ for all $i=1,...,n$ then
\begin{eqnarray*}
\sigma_\omega(x)&=&\sum_{i=1}^n\omega(v_i)\langle e_i,x\rangle\\
&=&\sum_{i=1}^n\sum_{j=1}^n\tau(v_j)u_{ji}\langle e_i,x\rangle\\
&=&\sigma_\tau\left(\sum_{i=1}^n\sum_{j=1}^ne_ju_{ji}\langle e_i,x\rangle\right)\\
&=&\sigma_\tau(Ux)\end{eqnarray*}
for all $x\in  A^n$.
\end{proof}
\subsection{Quasifree Equivalence}
It is a standard fact that when $W\in B(H)$ is a unitary the mapping $Ad_W:T\mapsto WTW^*$ is a $*$-automorphism of $B(H)$. Of more interest to us is that when $A\subset B(H)$ is a unital $C^*$-algebra and $Ad_W$ restricts to an automorphism of $A$ we will say that $W$ is \emph{$A$-fixing}. Since $Ad_{W^*}=(Ad_W)^{-1}$ we find that $W$ is $A$-fixing if and only if $W^*$ is $A$-fixing. It is also obvious that if $W_1$ and $W_2$ are both $A$-fixing then $W_1W_2$ is $A$-fixing as well.

If $\omega$ is a representation of $\E_n$ and $W\in B(H)$ is a unitary, denote by $Ad_W\circ \omega$ the representation generated by $v_i\mapsto W\omega(v_i)W^*$. 

\begin{definition}\label{qfeq} Two nondegenerate representations $\omega$ and $\tau$ (on $H$) are \emph{$A$-quasifree equivalent}, denoted $\omega\aqf\tau$, if there is an $A$-fixing unitary $W\in B(H)$ such that $\omega$ and $Ad_W\circ\tau$ are $A$-free equivalent.
\end{definition}

\begin{proposition}\label{qfeform} $\omega\aqf\tau$ if and only if there is an $A$-fixing unitary $W\in  B(H)$ and a unitary $U=[u_{jk}]\in M_n(A)$ for which
$$\omega(v_i)=\sum_{j=1}^nW\tau(v_j)W^*u_{ji}$$
for all $i=1,...,n$.
\end{proposition}
The proof is simply an application of Theorem \ref{thm1} to the representations $\omega$ and $Adj_W\circ\tau$.

\begin{proposition} For any given unital $C^*$-algebra $A\subseteq B(H)$, $A$-quasifree equivalence is an equivalence relation.
\end{proposition}
\begin{proof}
Reflexivity is obvious. 

Suppose that $\omega\aqf\tau$,which by Proposition \ref{qfeform} allows us to conclude that $\omega(v_i)=\sum_{j=1}^nW\tau(v_j)W^*u_{ji}$ for some unitaries $U=[u_{jk}]\in M_n(A)$ and $W\in B(H)$ (which is $A$-fixing) and all $i=1,...,n$. Set $v_{jk}=W^*u_{kj}^*W\in A$ and note that $V=[v_{jk}]\in M_n(A)$ is unitary. It remains to check that for each $i=1,...,n$
\begin{eqnarray*}
\sum_{j=1}^nW^*\omega(v_j)Wv_{ji}
&=&\sum_{j=1}^nW^*\left(\sum_{k=1}^nW\tau(v_k)W^*u_{kj}\right)Wv_{ji}\\
&=&\sum_{j=1}^nW^*\left(\sum_{k=1}^nW\tau(v_k)W^*u_{kj}\right)WW^*u_{ij}^*W\\
&=&\sum_{j=1}^n\sum_{k=1}^n\tau(v_k)W^*u_{kj}u_{ij}^*W\\
&=&\sum_{k=1}^n\tau(v_k)W^*\left(\sum_{j=1}^nu_{kj}u_{ij}^*\right)W\\
&=&\sum_{i=1}^n\tau(v_k)W^*\delta_{ki}W\\
&=&\tau(v_i)\end{eqnarray*}
and so, by the previous proposition, $\tau\aqf\omega$.

To prove transitivity, suppose that $\omega\aqf\tau$ and $\tau\aqf\kappa$, so by Proposition \ref{qfeform}
$$\omega(v_i)=\sum_{j=1}^nW_1\tau(v_j)W_1^*u_{ji}$$
$$\tau(v_j)=\sum_{k=1}^nW_2\kappa(v_k)W_2^*v_{kj}$$
for some $A$-fixing unitaries $W_1,W_2\in B(H)$ and unitaries $U=[u_{jk}],V=[v_{jk}]\in M_n(A)$ and all $i,j=1,...,n$. Thus
\begin{eqnarray*}
\omega(v_i)&=&\sum_{j=1}^nW_1\tau(v_j)W_1^*u_{ji}\\
&=&\sum_{j=1}^nW_1\left(\sum_{k=1}^nW_2\kappa(v_k)W_2^*v_{kj}\right)W_1^*u_{ji}\\
&=&\sum_{k=1}^nW_1W_2\kappa(v_k)(W_1W_2)^*\left(\sum_{j=1}^nW_1v_{kj}W_1^*u_{ji}\right).\end{eqnarray*}
If we let $t_{ki}:=\sum_{j=1}^nW_1v_{kj}W_1^*u_{ji}\in A$ then $T=[t_{ki}]\in M_n(A)$ is unitary since
\begin{eqnarray*}(T^*T)_{ij}=\sum_{k=1}t_{ki}^*t_{kj}&=&\sum_{k=1}^n\left(\sum_{l=1}^nW_1v_{kl}W_1^*u_{li}\right)^*\left(\sum_{m=1}^nW_1v_{km}W_1^*u_{mj}\right)\\
&=&\sum_{k=1}^n\sum_{l=1}^n\sum_{m=1}^nu_{li}^*W_1v_{kl}^*v_{km}W_1^*u_{mj}\\
&=&\sum_{l=1}^n\sum_{m=1}^nu_{li}^*W_1\left(\sum_{k=1}^nv_{kl}^*v_{km}\right)W_1^*u_{mj}\\
&=&\sum_{l=1}^n\sum_{m=1}^nu_{li}^*W_1\delta_{lm}W_1^*u_{mj}\\
&=&\sum_{m=1}^nu_{mi}^*u_{mj}\\
&=&\delta_{ij}I\end{eqnarray*}
and similarly tedious calculations show $(TT^*)_{ij}=\delta_{ij}$.

Denoting by $W'$ the unitary $W_1W_2\in B(H)$ (which is also $A$-fixing), we have for each $i=1,...,n$
$$\omega(v_i)=\sum_{k=1}^nW_1W_2\kappa(v_k)(W_1W_2)^*\left(\sum_{j=1}^nW_1v_{kj}W_1^*u_{ji}\right)=\sum_{j=1}^nW'\kappa(v_j)W'^*t_{ji}$$ and so $\omega\aqf\kappa$ as desired.
\end{proof}

\section{Equivalence and Multiplicity}
The complex subspace of $\E_n$ spanned by the generating isometries is a Hilbert space $E_n\cong \C^n$. 
\begin{definition} The \emph{multiplicity} of a nondegenerate $*$-representation $\pi:\E_n\to B(H)$ is the dimension of $(\pi(E_n)H)^\perp$.
\end{definition} 
Our goal for this section is the proof of the following theorem:
\begin{theorem}\label{multhm} Two representations of $\E_n$ on $H$ are $B(H)$-quasifree equivalent if and only if they have the same multiplicity.
\end{theorem}
We will accomplish this through several intermediate lemmas.

Let $p_n:=I-\sum_{i=1}^nv_iv_i^*$ and define $J_n\subset \E_n$ to be the (closed, two-sided) ideal generated by $p_n$. Note that the multiplicity of $\pi$ is also the rank of the projection $\pi(p_n)$. Representations of multiplicity $0$ (also called \emph{essential representations} in the literature) factor through the quotient $\E_n/J_n\cong \O_n$ and thus may be thought of as representations of the Cuntz algebra.

\begin{lemma}\label{lem1} Any two essential representations of $\E_n$ on $H$ are $B(H)$-free equivalent.
\end{lemma}
\begin{proof} If $\omega$ and $\tau$ are essential representations of $\E_n$ on a Hilbert space $H$ then necessarily 
$$\sum_{i=1}^n\omega(v_i)\omega(v_i)^*=I=\sum_{j=1}^n\tau(v_j)\tau(v_j)^*.$$
Hence $\omega(v_i)=\sum_{j=1}^n\tau(v_j)\tau(v_j)^*\omega(v_i)$ for each $i=1,...,n$. Defining $u_{jk}:=\tau(v_j)^*\omega(v_k)$ for each $j,k=1,...,n$  it is straightforward to check that 
$$\sum_{i=1}^nu_{ji}u_{ki}^*=\sum_{i=1}^n\tau(v_j)^*\omega(v_i)^*\omega(v_i)^*\tau(v_k)=\tau(v_j)^*I\tau(v_k)=\delta_{jk}I$$
and similarly $\sum_{i=1}^nu_{ij}^*u_{ik}=\delta_{jk}I$, hence $U=[u_{jk}]\in M_n(B(H))$ is unitary.
\end{proof}

The \emph{Fock representation} of $\E_n$ is a nondegenerate $*$-representation of multiplicity $1$ which reduces to the unilateral shift when $n=1$. We'll briefly outline the construction which is originally due to Evans \cite{evans}. Let $E_n^{\otimes j}$ denote the $j$-fold tensor product Hilbert space (where $E_n^{\otimes 0}:=\C$) and define the \emph{full Fock space}
$$F_E:=\bigoplus_{j=0}^\infty E_n^{\otimes j}.$$
For each $i=1,...,n$ and all $j\geq 0$ we'll form maps $\varphi(v_i):E_n^{\otimes j}\to E_n^{\otimes j+1}$ by $\varphi(v_i)x=e_i\otimes x$. Taken all together these define a family of linear maps $\{\varphi(v_i):F_E\to F_E\}$ which are easily seen to be isometries which have pairwise orthogonal ranges. Hence $v_i\mapsto\varphi(v_i)$ extends to a representation $\varphi:\E_n\to B(F_E)$. This is the canonical Fock representation and it is routine to see that $\varphi$ has multiplicity $1$.  The \emph{Fock representation of multiplicity $k$} is analogously defined but with $\varphi^k(v_i):E_n^{\otimes j}\to E_n^{\otimes j+k}$ given by $\varphi^k(v_i)x=(\varphi(v_i))^kx$ for each $i=1,...,n$.

\begin{lemma}\label{lem2} For any Toeplitz algebra $\E_n$ and integers $a,b\geq 1$, $\varphi^a$ is $B(F_E)$-free equivalent to $\varphi^b$ if and only if $a=b$.
\end{lemma}
\begin{proof} The ``if'' statement is trivial.

Suppose that $\varphi^a$ and $\varphi^b$ are $B(F_E)$-free equivalent. Then there is a unitary $U=[u_{jk}]\in M_n(B(F_E))$ such that
$$\varphi^a(v_i)=\sum_{j=1}^n\varphi^b(v_j)u_{ji}$$
for all $i=1,...,n$. Hence
\begin{eqnarray*}
\sum_{i=1}^n\varphi^a(v_i)\varphi^a(v_i)^*
&=&\sum_{i=1}^n\left(\sum_{j=1}^n\varphi^b(v_j)u_{ji}\right)\left(\sum_{k=1}^n\varphi^b(v_k)u_{ki}\right)^*\\
&=&\sum_{i=1}^n\sum_{j=1}^n\sum_{k=1}^n\varphi^b(v_j)u_{ji}u_{ki}^*\varphi^b(v_k)^*\\
&=&\sum_{j=1}^n\sum_{k=1}^n\varphi^b(v_j)\left(\sum_{i=1}^nu_{ji}u_{ki}^*\right)\varphi^b(v_k)^*\\
&=&\sum_{j=1}^n\sum_{k=1}^n\varphi^b(v_j)\delta_{jk}\varphi^b(v_k)^*\\
&=&\sum_{j=1}^n\varphi^b(v_j)\varphi^b(v_j)^*.\end{eqnarray*}

As a consequence we find that $$\varphi^a(p_n)=I-\sum_{i=1}^n\varphi^a(v_i)\varphi^a(v_i)^*=I-\sum_{j=1}^n\varphi^b(v_j)\varphi^b(v_j)^*=\varphi^b(p_n)$$ and thus $(\varphi^a(p_n)B(F_E))^\perp=(\varphi^b(p_n)B(F_E))^\perp$. But that means $\varphi^a$ and $\varphi^b$ share the same multiplicity.

\end{proof}

We will make use of a (rephrased) result of Popescu \cite[Theorem 1.3]{popescu} which generalized the Wold decomposition: every representation $\pi$ of $E_n$ has a decomposition $\pi=\pi_e\oplus\pi_s$ where $\pi_e$ is an essential representation and $\pi_s$ is unitarily equivalent to a multiple of the Fock representation.
\begin{lemma}\label{lem3} For $i=1,2$ suppose that $\omega_i$ and $\tau_i$ are nondegenerate representations of $\E_n$ on $H_i$; $A_i\subset B(H_i)$ are unital $C^*$-algebras; and $\omega_i\overset{A_i\textrm{-q.f.}}{\sim}\tau_i$. Then $\omega_1\oplus\omega_2$ is $A_1\oplus A_2$-quasifree equivalent to $\tau_1\oplus\tau_2$.
\end{lemma}
Naturally the statement may be generalized to more than two pairs of representations, but, because of the Wold decomposition analogue, this is enough for our purposes.
\begin{proof} Denote $\omega=\omega_1\oplus\omega_2$ and $\tau=\tau_1\oplus\tau_2$. Since $\omega_i\overset{A_i\textrm{-q.f.}}{\sim}\tau_i$ for $i=1,2$ we have unitaries $W_i\in B(H_i)$ which are, respectively, $A_i$-fixing as well as unitaries $U_i=[u_{jk}^{(i)}]\in M_n(A_i)$ for which
$$\omega_i(v_j)=\sum_{k=1}^nW_i\tau_i(v_k)W_i^*u_{kj}^{(i)}$$
for all $j=1,...,n$ and $i=1,2$.

It is clear that $W:=W_1\oplus W_2$ is a unitary in $B(H_1\oplus H_2)$ which fixes the $C^*$-algebra $A_1\oplus A_2$. Tedious but routine calculations demonstrate that $U:=[(u_{jk}^{(1)},u_{jk}^{(2)})]\in M_n(A_1\oplus A_2)$ is unitary. All that remains is to show that
$$\omega(v_k)=\sum_{j=1}^nW\tau(v_j)(U)_{jk}$$
for all $k=1,...,n$. Now for each $j,k=1,...,n$ we have
\begin{eqnarray*}
W\pi_2(v_j)W^*(U)_{jk}&=&\left(W_1\oplus W_2\right)(\tau_1(v_j)\oplus\tau_2(v_j))\left(W_1^*\oplus W_2^*\right)(u_{jk}^{(1)}\oplus u_{jk}^{(2)})\\
&=&\left(W_1\tau_1(v_j)W_1^*u_{jk}^{(1)}\right)\oplus\left(W_2\tau_2(v_j)W_2^*u_{jk}^{(2)}\right)\end{eqnarray*}
hence
\begin{eqnarray*}
\sum_{j=1}^nW\tau(v_j)W^*(U)_{ji}&=&\sum_{j=1}^n\left(W_1\tau_1(v_j)W_1^*u_{jk}^{(1)}\right)\oplus\left(W_2\tau_2(v_j)W_2^*u_{jk}^{(1)}\right)\\
&=&\left(\sum_{j=1}^nW_1\tau_1(v_j)W_1^*u_{jk}^{(1)}\right)\oplus\left(\sum_{j=1}^nW_2\tau_2(v_j)W_2^*u_{jk}^{(1)}\right)\\
&=&\omega_1(v_k)\oplus\omega_2(v_k)\\
&=&\omega(v_k)\end{eqnarray*}
as desired.
\end{proof}

We now have the tools to prove our theorem. 

\begin{proof}[Proof of Theorem \ref{multhm}] We shall first prove necessity. If $\omega$ and $\tau$ are $B(H)$-quasifree equivalent then
$$\omega(v_i)=\sum_{j=1}^nW\tau(v_j)W^*u_{ji}$$
for unitaries $W\in B(H)$ and $U=[u_{jk}]\in M_n(B(H))$. Thus for we have
\begin{eqnarray*}
\sum_{i=1}^n\omega(v_i)\omega(v_i)^*
&=&\sum_{i=1}^n\left(\sum_{j=1}^nW\tau(v_j)W^*u_{ji}\right)\left(\sum_{k=1}^nW\tau(v_k)WW^*u_{ki}^*\right)\\
&=&\sum_{i=1}^n\sum_{j=1}^n\sum_{k=1}^nW\tau(v_j)W^*u_{ji}u_{ki}^*W\tau(v_k)^*W^*\\
&=&\sum_{j=1}^n\sum_{k=1}^nW\tau(v_j)W^*\left(\sum_{i=1}^nu_{ji}u_{ki}^*\right)W\tau(v_k)^*W^*\\
&=&\sum_{j=1}^n\sum_{k=1}^nW\tau(v_j)W^*\delta_{jk}W\tau(v_k)^*W^*\\
&=&\sum_{j=1}^nW\tau(v_j)\tau(v_j)^*W^*\end{eqnarray*}
Thus $\omega(p_n)=W\tau(p_n)W^*$, where we recall that $p_n=I-\sum_{i=1}^nv_iv_i^*$. As a consequence $(\omega(E_n)H)^\perp=\omega(p_n)H$ and $\tau(p_n)H=(\tau(E_n)H)^\perp$ share the same dimension and we have our result.

Now suppose that $\omega$ and $\tau$ are nondegenerate $*$-representations of $\E_n$ on $H$ with the same multiplicity. By the analogue of the Wold decomposition we have $\omega=\omega_e\oplus\omega_s$ represented on $H=K_e\oplus K_s$ and $\tau=\tau_e\oplus \tau_s$ represented on $H=J_e\oplus J_s$. If the multiplicity is nonzero then $K_e$ and $J_e$ are nontrivial, separable, and consequently unitarily equivalent. Let $W_e\in B(J_e,K_e)$ be a unitary, then $W_e\tau_eW_e^*$ is an essential representation of $\E_n$ on $K_e$ and so is $B(K_e)$-free equivalent to $\omega_e$ by Lemma \ref{lem1}. Hence there is a unitary $V=[v_{jk}]\in M_n(B(K_e))$
$$\omega_e(v_i)=\sum_{j=1}^nW_e\tau_e(v_j)W_e^*v_{ji}$$
for all $i=1,...,n$.

Since $\omega$ and $\tau$ have the same multiplicity, $\omega_s$ and $\tau_s$ are both unitarily equivalent to the same multiple of the Fock representation, hence are unitarily equivalent to each other. Let $W_s\in B(J_s,K_s)$ be a unitary for which $\omega_s=W_s\tau_sW_s^*$.

Note first that, as $K_e\oplus K_s=H=J_e\oplus J_s$, we have that $W:=W_e\oplus W_s$ is a unitary in $B(H)$. Set $u_{jk}:=v_{jk}I\oplus\delta_{jk}I\in B(K_e)\oplus B(K_s)$ and $U=[u_{jk}]\in M_n(B(K_e)\oplus B(K_s))\subset M_n(B(H))$. Tedious but straightforward calculations demonstrate that $U$ is a unitary matrix. All that remains is to see that for each $i=1,...,n$
\begin{eqnarray*}
\sum_{j=1}^nW\tau(v_j)W^*u_{ji}&=&\sum_{j=1}^n(W_e\oplus W_s)(\tau_e(v_j)\oplus\tau_s(v_j))(W_e\oplus W_s)^*(v_{ji}\oplus\delta_{ji}I)\\
&=&\left(\sum_{j=1}^nW_e\tau_e(v_j)W_e^*v_{ji}\right)\oplus\left(\sum_{j=1}^nW_s\tau_s(v_j)W_s^*\delta_{ji}I\right)\\
&=&\omega_e(v_i)\oplus W_s\tau_s(v_i)W_s^*\\
&=&\omega_e(v_i)\oplus\omega_s(v_i)\\
&=&\omega(v_i)\end{eqnarray*}
and so $\omega$ is $B(H)$-quasifree equivalent to $\tau$.
\end{proof}

\begin{corollary} If $\omega\aqf\tau$ for any $A\subset B(H)$ then they share the same multiplicity.
\end{corollary}
\begin{proof} Any unitary $U\in M_n(A)$ is also a unitary in $M_n(B(H))$ and any $A$-fixing unitary $W\in B(H)$ is clearly $B(H)$-fixing. Hence $A$-quasifree equivalence implies $B(H)$-quasifree equivalence and so Theorem \ref{multhm} applies.
\end{proof}
That the converse of the corollary is false is seen in the results of the next section, where it is shown that quasifree equivalence is also related to the conjugacy of dynamical systems.

\section{Application to Dynamics}
Our goal is to relate free and quasifree equivalence of representations to conjugacy of certain $*$-endomorphisms of concretely represented $C^*$-algebras. To do so we will need the property of \emph{Invariant Basis Number (IBN)} for $C^*$-algebras and some related results. 

\begin{definition} A unital $C^*$-algebra $A$ has IBN if $A^n\simeq A^m$ if and only if $n=m$.
\end{definition}
Our main tool for detecting IBN is $K$-theoretical and is a result of the author's Ph.D.\ dissertation.
\begin{theorem}[Theorem 3.10 in \cite{gipson}] $A$ has IBN if and only if $[1_A]\in K_0(A)$ has finite additive order.
\end{theorem}
While nontrivial to check for arbitrary $C^*$-algebras, it is an immediate consequence of the theorem that a factor has IBN if and only it is finite.

Our main results for this section, Theorems \ref{thm1} and \ref{thm2}, relate the equivalence of representations of Toeplitz algebras to the conjugacy of certain $*$-endomorphisms they induce, and vice versa. Theorem \ref{thm1} is technically a special case of Theorem \ref{thm2}, but the proof of the later is markedly simpler if we have already established the former.

\begin{theorem}\label{thm1} Let $A\subseteq B(H)$ be a $C^*$-algebra, $\omega$ and $\tau$ nondegenerate $*$-representations of $\E_n$ and $\E_m$, respectively, ($n,m<\infty$) on $H$ such that
$$\alpha(a):=\sum_{i=1}^n\omega(v_i)a\omega(v_i)^*$$
$$\beta(a):=\sum_{j=1}^m\tau(v_j)a\tau(v_j)^*$$
are $*$-endomorphisms of $A$. If the relative commutant $A'$ has Invariant Basis Number then the following are equivalent:
\begin{enumerate}
\item $\alpha=\beta$,
\item $n=m$ and $\omega$ and $\tau$ are $A'$-free equivalent.
\end{enumerate}
\end{theorem}
\begin{proof} $1)\Rightarrow 2)$. Consider the subspaces
$$E_\alpha:=\{X\in B(H):Xa=\alpha(a)X\ for\ all\ a\in A\}$$
$$E_\beta:=\{X\in B(H):Xa=\beta(a)X\ for\ all\ a\in A\}$$
which are right $C^*$-modules over $A'$ when given the inner product $\langle x,y\rangle=x^*y$. Of course when $\alpha=\beta$ we have $E_\alpha=E_\beta$ and so, to avoid confusion, we'll call this $A'$-module $E$. Note that $\{\omega(v_i):i=1,...,n\}$ and $\{\tau(v_j):j=1,...,m\}$ are orthonormal sets in $E$ and denote by $E_\omega$ and $E_\tau$ the submodules of $E$ spanned by $\{\omega(v_i)\}$ and $\{\tau(v_j)\}$, respectively. Since $E_\omega$ has a basis of size $n$ it is unitarily equivalent to $(A')^n$, and similarly $E_\tau\simeq (A')^m$.
As $$\sum_{i=1}^n\omega(v_i)\omega(v_i)^*=\alpha(I)=\beta(I)=\sum_{j=1}^m\tau(v_j)\tau(v_j)^*$$
we thus have that $\omega(v_i)=\sum_{j=1}^m\tau(v_j)\tau(v_j)^*\omega(v_i)$ for all $i=1,...,n$ and $\tau(v_j)=\sum_{i=1}^n\omega(v_i)\omega(v_i)^*\tau(v_j)$ for all $j=1,...,m$. Hence $\omega(v_i)\in E_\tau$ (because $\tau(v_j)^*\omega(v_k)\in A'$ always) for all $i=1,...,n$ and $\tau(v_j)\in E_\omega$ for all $j=1,...,m$, thus $E_\omega=E_\tau$. But then $(A')^n\simeq E_\omega=E_\tau\simeq (A')^m$ and so, because $A'$ has IBN, we conclude that $n=m$.

Since $E_\tau=E_\omega$ and, by definition, $E_\omega=\sigma_\omega((A')^n)$ there are vectors $f_1,...,f_n\in (A')^n$ such that $\sigma_\omega(f_i)=\tau(v_i)$ for each $i=1,...,n$. As $\{\tau(v_i)\}$ is an orthonormal basis for $E_\omega$, and $\sigma_\omega$ is Toeplitz by construction, we conclude that $\{f_1,...,f_n\}$ is an orthonormal basis for $(A')^n$, hence there is a unitary $U=[u_{jk}]\in M_n(A')$ for which $Ue_i=f_i$ for each $i=1,...,n$. Thus $\tau(v_i)=\sigma_\omega(f_i)=\sigma_\omega(Ue_i)=\sum_{j=1}^n\omega(v_j)u_{ji}$ for each $i=1,...,n$ and hence $\tau$ is $A'$-free equivalent to $\omega$.

$2)\Rightarrow 1)$. If $\omega$ and $\tau$ are $A'$-free equivalent then $\omega(v_i)=\sum_{j=1}^n\tau(v_j)u_{ji}$ for some unitary $U=[u_{jk}]\in M_n(A')$ and each $i=1,...,n$. For all $a\in A$ we then have that
\begin{eqnarray*}
\alpha(a)&=&\sum_{i=1}^n\omega(v_i)a\omega(v_i)^*\\
&=&\sum_{i=1}^n\sum_{j=1}^n\sum_{k=1}^n\tau(v_j)u_{ji}au_{ki}^*\tau(v_k)^*\\
&=&\sum_{j=1}^n\sum_{k=1}^n\tau(v_j)a\left(\sum_{i=1}^nu_{ji}u_{ki}^*\right)\tau(v_k)^*\\
&=&\sum_{j=1}^n\sum_{k=1}^n\tau(v_j)a\delta_{jk}\tau(v_k)^*\\
&=&\sum_{j=1}^n\tau(v_j)a\tau(v_j)^*\\
&=&\beta(a)\end{eqnarray*}
hence $\alpha=\beta$.
\end{proof}

\begin{theorem}\label{thm2} Let $A\subseteq B(H)$ be a $C^*$-algebra, $\omega$ and $\tau$ nondegenerate $*$-representations of $\E_n$ and $\E_m$ (respectively, $n,m<\infty$) on $H$ such that
$$\alpha(a):=\sum_{i=1}^n\omega(v_i)a\omega(v_i)^*$$
$$\beta(a):=\sum_{j=1}^m\tau(v_j)a\tau(v_j)^*$$
are $*$-endomorphisms of $A$. If the relative commutant $A'$ has Invariant Basis Number then the following are equivalent:
\begin{enumerate}
\item $\alpha=Ad_{W}\circ\beta\circ Ad_{W^*}$ for some $A$-fixing unitary $W\in B(H)$,
\item $n=m$ and $\omega$ and $\tau$ are $A'$-quasifree equivalent.
\end{enumerate}
\end{theorem}
\begin{proof} For any $A$-fixing unitary $W\in B(H)$, letting $\kappa(v_i)=W\tau(v_i)W^*$ for each $i=1,...,n$ we have that $\kappa$ extends to a nondegenerate $*$-representation of $\E_m$. Note that 
$$\gamma(a):=\sum_{i=1}^m\kappa(v_i)a\kappa(v_i)^*
=W\sum_{i=1}^m\tau(v_i)W^*aW\tau(v_i)^*W^*=Ad_{W}\circ\beta\circ Ad_{W^*}$$
and so $\gamma$ is a $*$-endomorphism of $A$. 

$1)\Rightarrow 2).$ Let $W$ be from the hypothesis and construct $\kappa$ and $\gamma$ as above. Applying Theorem \ref{thm1} to $\gamma$ and $\alpha$ we conclude that $n=m$ and there is a unitary $U=[u_{jk}]\in M_n(A')$ such that $\omega(v_i)=\sum_{j=1}^n\kappa(v_i)u_{ji}$ for each $i=1,...,n$. But then we have $\omega(v_i)=\sum_{j=1}^nW\tau(v_i)W^*u_{ji}$ for each $i=1,...,n$ and hence $\omega\apqf\tau$.

$2)\Rightarrow 1).$ We have that 
$$\omega(v_i)=\sum_{j=1}^nW\tau(v_j)W^*u_{ji}=\sum_{j=1}^n\kappa(v_j)u_{ji}$$
for appropriate unitary $U=[u_{jk}]\in M_n(A')$ and an $A$-fixing unitary $W\in B(H)$, where $\kappa$ (and so $\gamma$) is constructed as before. Hence $\omega$ and $\kappa$ are $A'$-free equivalent and we apply Theorem \ref{thm1} to conclude that $\alpha=\gamma=Ad_{W}\beta Ad_{W^*}$ as desired.
\end{proof}

\section{Recovery of Known Results}
Particular choices for the $C^*$-algebra $A$ in Theorem \ref{thm1} and Theorem \ref{thm2} allow us to recover several previously known results.

\subsection{Endomorphisms of $B(H)$} First consider the case when $A=B(H)$. It was first observed by Arveson \cite{arveson} that every $*$-endomorphism of $B(H)$ is of the form $x\mapsto\sum_{i=1}^nV_ixV_i^*$ where $V_1,...,V_n$ (with $n=\infty$ a possibility) is some family of mutually orthogonal isometries. Laca \cite{laca} improves upon this observation by relating the conjugacy properties of endomorphisms of $B(H)$ to properties of the induced representations of Toeplitz algebras. Our Theorem \ref{thm2} recovers one of Laca's results when $n<\infty$.

Recall that the $\C$-linear span of the generators in $\E_n$ forms a Hilbert space $E_n\cong \C^n$. A unitary transformation $U$ of $E_n$ extends uniquely to an automorphism $\gamma_U$ of $\E_n$. These are termed ``quasifree'' automorphisms of $\E_n$ by Laca. We'd like to remark that the relationship $\omega=\tau\circ\gamma_U$, which we term the $\C$-quasifree equivalence of $\omega$ and $\tau$, is termed ``quasifree equivalence'' by Laca \cite[Definition 2.5]{laca}.

\begin{theorem}[Proposition 2.4 in \cite{laca}] Suppose that $\omega$ and $\tau$ are nondegenerate representations of $\E_n$ and $\E_m$, respectively, on a Hilbert space $H$. The corresponding endomorphisms 
$$\alpha(a):=\sum_{i=1}^n\omega(v_i)a\omega(v_i)^*$$
$$\beta(a):=\sum_{j=1}^m\tau(v_j)a\tau(v_j)^*$$
are conjugate if and only if $n=m$ and there is a unitary $U$ on $E_n$ for which $\omega$ and $\tau\circ\gamma_U$ are unitarily equivalent.
\end{theorem}
First, we know $B(H)'=\C$ has IBN. Second, every automorphism of $B(H)$ is inner and so $\alpha$ is conjugate to $\beta$ if and only if $\alpha=Ad_{W}\beta Ad_{W^*}$ for some unitary $W\in B(H)$. 
Finally, if $U=[u_{jk}]\in M_n(\C)=B(E_n)$ is unitary then $\gamma_U(v_i)=Uv_i=\sum_{j=1}^nv_ju_{ji}$ and so
$$\tau\circ\gamma_U(v_i)=\tau(Uv_i)=\tau\left(\sum_{j=1}^nv_ju_{ji}\right)=\sum_{j=1}^n\tau(v_j)u_{ji}.$$
Consequently $\omega=Ad_W\circ\tau\circ\gamma_U$ if and only if
$$\omega(v_i)=W\tau(Uv_i)W^*=W\sum_{j=1}^nv_ju_{ji}W^*=\sum_{j=1}^nWv_jW^*u_{ji}$$
(note the $u_{jk}$ are scalar) which happens if and only if $\omega\cqf\tau$.

Thus we see that our Theorem \ref{thm2} recovers Laca's result precisely when $A=B(H)$.

\subsection{Endomorphisms of $\textrm{II}_1$ Factors} Suppose that $A=M\subset B(H)$ is a type $\textrm{II}_1$ factor whose commutant is finite. This is the case when, for example, $M$ is represented standardly on $L^2(M,tr)$. Enomoto and Watatani \cite{enomoto} have proven, provided the Jones index is a positive integer $n<\infty$, that every unital $*$-endomorphism of $M$ is of the form $x\mapsto\sum_{i=1}^nV_ixV_i^*$ where $V_1,...,V_n$ is a Toeplitz family. They have also demonstrated relationships between conjugacy of endomorphisms and the corresponding representations of Toeplitz algebras, of which two results are special cases of our Theorems \ref{thm1} and \ref{thm2}.

\begin{theorem}[Theorem 6 in \cite{enomoto}] Let $M\subset B(H)$ be a $\textrm{II}_1$-factor with finite commutant. Let $\{s_1,...,s_n\}$ and $\{t_1,...,t_m\}$ be two Toeplitz families on $H$ and define endomorphisms $\alpha(x)=\sum_{i=1}^ns_ixs_i^*$ and $\beta(x)=\sum_{j=1}^mt_jxt_j^*$. Then $\alpha=\beta$ if and only if $n=m$ and there is a unitary matrix $U=[u_{jk}]\in M'\otimes M_n(\C)$ such that $t_i=\sum_{j=1}^ns_ju_{ji}$.
\end{theorem}
If we can conclude that $M'$ has IBN then the result is precisely our Theorem \ref{thm1} applied to the representations $\omega:v_i\mapsto s_i$ and $\tau:v_j\mapsto t_j$. By hypothesis $M'$ is a finite factor and thus $K_0(M')$ is either $\Z$ or $\R$. In either case we conclude that $M'$ has IBN.

\begin{theorem}[Theorem 8 in \cite{enomoto}] Let $M$ be a type $\textrm{II}_1$ factor acting standardly on a Hilbert space $H$. Let $\{s_1,...,s_n\}$ and $\{t_1,...,t_m\}$ be two Toeplitz families on $H$ and define endomorphisms $\alpha(x)=\sum_{i=1}^ns_ixs_i^*$ and $\beta(x)=\sum_{j=1}^mt_jxt_j^*$. Then $\alpha$ and $\beta$ are conjugate if and only if there are unitaries $W\in B(H)$ and $U=[u_{jk}]\in M'\otimes M_n(\C)$ such that $t_i=\sum_{j=1}^nWs_jW^*u_{ji}$.
\end{theorem}
This is precisely our second Theorem. A key insight, noted in \cite{enomoto}, is that any automorphism of a $\textrm{II}_1$-factor $M$ represented standardly is the restriction of an inner automorphism of $B(H)$. Hence $\alpha$ is conjugate to $\beta$ if and only if $\alpha=Ad_{W}\circ\beta\circ Ad_{W^*}$ for some unitary $W\in B(H)$ which is $M$-fixing. If $WMW^*=W^*MW=M$ then for any $x\in M'$ and $y\in M$ we have $(WxW^*)y=WxW^*yWW^*=WW^*yWxW^*=yWxW^*$ and similarly $(W^*xW)y=y(W^*xW)$. Thus $W$ is also $M'$-fixing.

\bibliographystyle{plain}
\bibliography{bibl}

\begin{thebibliography}{1}

\bibitem{arveson}
W.~Arveson.
\newblock Continuous analogues of {F}ock space.
\newblock {\em Mem. Amer. Math. Soc.}, 80(409):iv+66, 1989.

\bibitem{cuntz}
J.~Cuntz.
\newblock Simple {$C\sp*$}-algebras generated by isometries.
\newblock {\em Comm. Math. Phys.}, 57(2):173--185, 1977.

\bibitem{enomoto}
M.~Enomoto and Y.~Watatani.
\newblock Endomorphisms of type {${\rm II}_1$}-factors and {C}untz algebras.
\newblock {\em J. Austral. Math. Soc. Ser. A}, 60(3):343--354, 1996.

\bibitem{evans}
D.~E. Evans.
\newblock On {$O_{n}$}.
\newblock {\em Publ. Res. Inst. Math. Sci.}, 16(3):915--927, 1980.

\bibitem{gipson}
P.~M. Gipson.
\newblock {\em Invariant Basis Number and Basis Types for {$C^*$}-Algebras}.
\newblock PhD thesis, University of Nebraska - Lincoln.

\bibitem{laca}
M.~Laca.
\newblock Endomorphisms of {$ B( H)$} and {C}untz algebras.
\newblock {\em J. Operator Theory}, 30(1):85--108, 1993.

\bibitem{popescu}
G.~Popescu.
\newblock Isometric dilations for infinite sequences of noncommuting operators.
\newblock {\em Trans. Amer. Math. Soc.}, 316(2):523--536, 1989.

\end{thebibliography}

\end{document}